\documentclass[11pt]{article}
\usepackage[utf8]{inputenc}
\usepackage[T1]{fontenc}
\usepackage{mathpazo}
\usepackage{amsmath,amssymb,amsthm,mathrsfs}
\usepackage{geometry}
\usepackage{tikz-cd}
\usepackage{enumitem}
\usepackage{hyperref}

\newtheorem{theorem}{Theorem}[section]
\newtheorem{proposition}[theorem]{Proposition}
\newtheorem{definition}[theorem]{Definition}
\newtheorem{lemma}[theorem]{Lemma}
\newtheorem{remark}[theorem]{Remark}
\newtheorem{corollary}[theorem]{Corollary}

\newcommand{\RiemIso}{\mathsf{Riem}_{\mathrm{iso}}}

\newcommand{\SingFlows}{\mathsf{SingFlows}} 
\newcommand{\Dec}{\mathsf{Dec}}
\newcommand{\Met}{\mathrm{Met}}
\newcommand{\Diff}{\mathrm{Diff}}

\title{\textbf{RICCI FLOW AS FUNCTOR}}
\author{Alexander Plakhotnikov \footnote{St. Petersburg State University, Department of Mathematics and Mechanics. 7-9 Universitetskaya Embankment, St Petersburg, Russia, 199034.\\
Email adress: \url{st132770@student.spbu.ru}}}
\date{\today}

\begin{document}
\maketitle
\begin{abstract}
    In this note we attempt to propose a categorical framework for the Ricci flow, treating it as a sequence of functors connecting the stack of Riemannian metrics to the category of geometric decompositions via singular flow spacetimes. To rigorize the domain of the flow, we adapt the definition of differentiable stacks to the site of Banach manifolds. We demonstrate that the Ricci flow defines a stratification of this stack.
\end{abstract}
\small{\tableofcontents}

\medskip

\section{Introduction}
This note addresses the following problem: Can we formalize the Ricci flow as a functor between algebraic categories? Such a formalization allows us to treat the preservation of symmetries and the stability of geometric limits using the language of algebraic geometry and stack theory.

We construct a hierarchical sequence of functors:
\begin{equation} \label{eq:hierarchy}
\begin{tikzcd}
    \RiemIso \arrow[r, "\mathcal{F}"] &
    \SingFlows \arrow[r, "\mathcal{A}"] & \Dec
\end{tikzcd}
\end{equation}
where $\RiemIso$ is the category of Riemannian manifolds, $\SingFlows$ is the category of singular Ricci flow spacetimes, and $\Dec$ is the category of geometric decompositions.

Our main contributions are the following three theorems:

\begin{theorem}[Theorem A: Functorial Construction]
The assignment of a singular Ricci flow spacetime to a compact Riemannian 3-manifold constitutes a functor $\mathcal{F}: \RiemIso \longrightarrow\SingFlows$. Furthermore, the asymptotic limit of the flow defines a functor $\mathcal{A}: \SingFlows \longrightarrow\Dec$. The composition $\mathcal{D} = \mathcal{A} \circ \mathcal{F}$ provides a canonical geometric decomposition for any initial metric.
\end{theorem}

To handle the diffeomorphism invariance rigorously, we model the space of initial data as an infinite-dimensional differentiable stack $\mathcal{X} = [\Met(M) / \Diff(M)]$.

\begin{theorem}[Theorem B: Stratification of the Moduli Stack]
The functor $\mathcal{D}$ induces a stratification of the stack $\mathcal{X}$. Specifically, if $g_0$ is a metric such that the flow converges to a hyperbolic limit, then there exists an open substack $\mathcal{U} \subset \mathcal{X}$ (in the Fr\'echet topology) containing $[g_0]$ such that every metric in $\mathcal{U}$ flows to the same hyperbolic limit.
\end{theorem}

\begin{theorem}[Theorem C: Conservation of Symmetries]
Let $(M, g)$ be a Riemannian manifold and $\mathrm{Iso}(g)$ its group of isometries. There is an injective group homomorphism
\[ \Psi: \mathrm{Iso}(g) \hookrightarrow \mathrm{Aut}(\mathcal{M}), \]
where $\mathcal{M}$ is the singular Ricci flow spacetime generated by $g$. This implies that symmetries of the initial data are preserved as global automorphisms of the spacetime, even across topological singularities.
\end{theorem}

\paragraph{Acknowledgements.} I would like to express my deep gratitude to Gregory Taroyan for the criticism of ideas, valuable advices and fruitful discussions.

\section{Preliminaries}

\subsection*{Sobolev Spaces for Metrics}
To treat the space of metrics as a manifold, we must fix the regularity. We follow Hebey's standard construction.

\begin{definition}[\cite{Heb96} Definition 2.1]
Let $(M, g)$ be a Riemannian manifold. The Sobolev space $H^p_k(M)$ is the completion of the space of smooth functions $C^\infty(M)$ with respect to the norm:
\[ \|u\|_{H^p_k} = \sum_{j=0}^k \left(\int_M |\nabla^j u|^p \, dv_g \right)^{1/p}. \]
\end{definition}

We fix $p=2$ and $k > n/2 + 2$ so that metrics are $C^2$-continuous and the Sobolev embedding theorems hold.

\subsection*{Banach Manifolds}

Since the space of metrics is infinite-dimensional, we require the definition of a Banach manifold. 

\begin{definition}[see, e.g. \cite{Ban16} Chapter I]
Let $E$ be a Banach space. A topological space $X$ is called a Banach manifold if there exists an atlas $\{(U_i, \varphi_i)\}_{i \in I}$ such that:
\begin{enumerate}
    \item $\{U_i\}$ is an open cover of $X$.
    \item For each $i$, $\varphi_i: U_i \longrightarrow\varphi_i(U_i) \subset E$ is a homeomorphism onto an open subset of $E$.
    \item For any $i, j$ with $U_i \cap U_j \neq \emptyset$, the transition map $\varphi_j \circ \varphi_i^{-1}: \varphi_i(U_i \cap U_j) \longrightarrow\varphi_j(U_i \cap U_j)$ is a $C^\infty$-morphism of Banach spaces.
\end{enumerate}
\end{definition}

In our context, the space $\Met^k(M)$ of Riemannian metrics of Sobolev class $H^k$ is a Banach manifold modeled on the Banach space of symmetric 2-tensors sections $H^k(S^2 T^*M)$.

\subsection*{Infinite-Dimensional Differentiable Stacks}

To define the stack of metrics, we must use the language of differentiable stacks. However, standard definitions in the literature are restricted to finite dimensions.

C. Blohmann defines a differentiable stack as follows:

\begin{definition}[\cite{Blo08} Definition 2.3]
A stack $\mathcal{X}$ over the site of finite-dimensional manifolds $\mathfrak{Mfld}$ is called differentiable if there is a manifold $S \in \mathfrak{Mfld}$ and a surjective representable submersion $S \longrightarrow\mathcal{X}$, called an atlas.
\end{definition}

We motivate our adaptation as follows: the object of our study is the moduli stack of Riemannian metrics. The space of sections $\Met^k(M)$ is a Banach manifold, not a finite-dimensional manifold, consequently, it does not belong to the category $\mathfrak{Mfld}$, and the standard definition of a differentiable stack is inapplicable. To rigorously treat the Ricci flow on the space of metrics, we must extend the site from finite-dimensional manifolds to the category of Banach manifolds, denoted $\mathbf{_M Ban}$.

We propose the following definition for use in infinite-dimensional case:

\begin{definition}
Let $\mathbf{_M Ban}$ be the category of Banach manifolds equipped with the open cover topology. A category fibered in groupoids $\mathcal{X}$ over $\mathbf{_M Ban}$ is called an infinite-dimensional differentiable stack if:
\begin{enumerate}
    \item It satisfies the descent axioms for isomorphisms and objects with respect to open covers in $\mathbf{_M Ban}$. \,\footnote{analogous to \cite[Definition 2.2]{Blo08} }
    \item It admits an atlas: there exists a Banach manifold $S \in \mathbf{_M Ban}$ and a surjective representable submersion $S \longrightarrow \mathcal{X}$.
\end{enumerate}
\end{definition}
\begin{remark}
Let $M$ be a closed manifold. Let $X = \Met^k(M)$ be the Banach manifold of $H^k$ metrics. Let $G = \Diff^{k+1}(M)$ be the Hilbert-Lie group of $H^{k+1}$ diffeomorphisms. \footnote{see \cite[Theorem 43.1]{KM97} }
The action of $G$ on $X$ defines a translation Lie groupoid $G \times X \rightrightarrows X$. The associated quotient stack
\[ \mathcal{X}_{\mathrm{Ric}} = [\Met^k(M) / \Diff^{k+1}(M)] \]
is an infinite-dimensional differentiable stack according to the definition above, with atlas $X \longrightarrow\mathcal{X}_{\mathrm{Ric}}$.
\end{remark}
\begin{definition}
A singular Ricci flow $\mathcal{M}$ is said to have a hyperbolic limit if, as the time function $\mathfrak{t} \to \infty$, the thick part of the spatial slices converges in the smooth topology to a complete hyperbolic 3-manifold $(H, g_{\mathrm{hyp}})$ of finite volume with constant sectional curvature $-1$.
\end{definition}
\begin{definition}
Let $\mathcal{X}$ be a stack over $\mathbf{_M Ban}$.
\begin{enumerate}
    \item A substack $\mathcal{Y} \subset \mathcal{X}$ is a full subcategory of $\mathcal{X}$ that is closed under isomorphisms and pullbacks. That is, if an object $x$ over $U$ belongs to $\mathcal{Y}$, then any object isomorphic to $x$ belongs to $\mathcal{Y}$, and for any smooth map $f: V \to U$, the pullback $f^*x$ belongs to $\mathcal{Y}$.
    \item A substack $\mathcal{U} \subset \mathcal{X}$ is called open if the inclusion morphism $\iota: \mathcal{U} \hookrightarrow \mathcal{X}$ is representable by open embeddings. Explicitly, for any Banach manifold $T$ and any morphism $T \to \mathcal{X}$, the fiber product $T \times_{\mathcal{X}} \mathcal{U}$ is represented by an open submanifold of $T$. In the context of an atlas $S \to \mathcal{X}$, $\mathcal{U}$ corresponds to a $G$-invariant open subset of $S$.
\end{enumerate}
\end{definition}
To interpret the Ricci flow as a gradient flow on the moduli stack, we recall the variational structure introduced by Perelman.
\begin{definition}[\cite{Per02} Section 1.1 and proof of Proposition 1.2]
Let $(M,g)$ be a closed Riemannian manifold and $f \in C^\infty(M)$ be a scalar function. The $\mathcal{F}$-functional is defined as:
\[ \mathcal{F}(g, f) = \int_M (R_g + |\nabla f|^2) e^{-f} \, dv_g. \]
Consider the measure $dm = e^{-f} dv_g$ as fixed. The gradient flow of $\mathcal{F}$ with respect to the $L^2$-metric on the space of metrics is given by the system:
\begin{equation} \label{eq:perelman_flow}
\begin{cases}
\partial_t g_{ij} &= -2(R_{ij} + \nabla_i \nabla_j f), \\
\partial_t f &= -R - \Delta f + |\nabla f|^2.
\end{cases}
\end{equation}
This system is strictly parabolic modulo diffeomorphisms generated by $\nabla f$.
\end{definition}
\section{Categorical framework}
\subsection*{The Smooth Flow and Stack Trajectories}
The Ricci flow equation $\partial_t g = -2\mathrm{Ric}(g)$ defines a curve in the Banach manifold $\Met(M)$. 
\begin{theorem}[\cite{DeT83} Theorem 1]
For any compact Riemannian manifold $(M, g_0)$, there exists a unique smooth solution $g(t)$ to the Ricci flow defined on a maximal interval $[0, T)$.
\end{theorem}
\begin{proposition}
The Ricci flow defines a global section of the tangent bundle of the moduli stack $\mathcal{X}_{\mathrm{Ric}} = [\Met(M) / \Diff(M)]$.
\end{proposition}

\begin{proof}
Recall that for a quotient stack $\mathcal{X} = [U / G]$, the tangent stack $T\mathcal{X}$ is presented by the transformation groupoid associated with the tangent action of $G$ on $T U$. A vector field on the stack $\mathcal{X}$ is defined as a section of the projection $T\mathcal{X} \to \mathcal{X}$. In terms of the presentation, this corresponds to a smooth vector field $V: U \to TU$ on the atlas $U$ that is $G$-equivariant.

Let $U = \Met(M)$ and $G = \Diff(M)$. The action of $\varphi \in G$ on $g \in U$ is given by the pullback $R_\varphi(g) = \varphi^* g$. The tangent space at a point $g$, denoted $T_g \Met(M)$, is identified with the space of symmetric $(0,2)$-tensor fields $\Gamma(S^2 T^*M)$. The differential of the action, $d(R_\varphi): T_g \Met(M) \to T_{\varphi^* g} \Met(M)$, acts linearly on tensors: for any $h \in T_g \Met(M)$, we have $d(R_\varphi)(h) = \varphi^* h$.

The Ricci flow is defined by the vector field $\mathcal{R}: \Met(M) \to T\Met(M)$ given by $\mathcal{R}(g) = -2\mathrm{Ric}(g)$. To show that $\mathcal{R}$ descends to the stack, we must verify the equivariance condition:
\[ \mathcal{R}(R_\varphi(g)) = d(R_\varphi)(\mathcal{R}(g)) \]
for all $g \in \Met(M)$ and $\varphi \in \Diff(M)$.

Evaluating the left-hand side:
\[ \mathcal{R}(\varphi^* g) = -2\mathrm{Ric}(\varphi^* g). \]
Evaluating the right-hand side:
\[ d(R_\varphi)(-2\mathrm{Ric}(g)) = -2\varphi^* (\mathrm{Ric}(g)). \]
By the naturality of the Ricci curvature tensor (it is a geometric invariant), we have the identity $\mathrm{Ric}(\varphi^* g) = \varphi^* \mathrm{Ric}(g)$. Thus, the left and right sides are equal. 

This implies that the assignment $g \mapsto -2\mathrm{Ric}(g)$ is compatible with the groupoid structure and defines a well-posed global section of $T\mathcal{X}_{\mathrm{Ric}}$. Unlike the classical quotient space $\Met(M)/\Diff(M)$, where the vector field is ill-defined at metrics with isometries (due to the non-free action), the stack formalism retains the isotropy information, making the equivariant vector field the correct rigorous object.
\end{proof}
\subsection*{Singularities and Spacetimes}
The smooth flow curve $g(t)$ may only exist for a finite time $T < \infty$. To continue the flow, we adopt the spacetime formalism.
\paragraph{The Category $\SingFlows$}
\begin{definition}[\cite{BK19} Definition 1.2]
A singular Ricci flow is a tuple $(\mathcal{M}, \mathfrak{t}, \partial_{\mathfrak{t}}, g)$ where $\mathcal{M}$ is a smooth 4-manifold (possibly with boundary), $\mathfrak{t}$ is a time function, and $g$ is a metric on spatial slices satisfying the Ricci flow equation and the Canonical Neighborhood Assumption.\,\footnote{as defined in \cite[Definition 5.7]{BK19} }
\end{definition}
\begin{definition}[Category of Singular Flows]
The category $\SingFlows$ consists of:
\begin{itemize}
    \item \textbf{Objects:} Singular Ricci flow spacetimes satisfying the completeness and canonical neighborhood conditions.
    \item \textbf{Morphisms:} Smooth embeddings $\Phi: \mathcal{M} \longrightarrow\mathcal{M}'$ that preserve the time function ($\mathfrak{t}' \circ \Phi = \mathfrak{t}$), the time vector field, and the metric tensor ($g = \Phi^* g'$).
\end{itemize}
\end{definition}
\begin{remark}
In $\SingFlows$, the object $\mathcal{M}$ is 4-dimensional. The spatial topology $M_t = \mathfrak{t}^{-1}(t)$ changes at singular times. This internalizes the surgery process into the geometric definition of the object.
\end{remark}
\begin{proposition}
There exists a functor $\mathcal{C}: \RiemIso \longrightarrow\SingFlows$ for closed 3-manifolds.
\end{proposition}
\begin{proof}
We construct the functor $\mathcal{C}$ as follows:
\begin{enumerate}
\item On Objects:

For a Riemannian manifold $(M, g) \in \RiemIso$, we define $\mathcal{C}(M, g)$ to be the unique maximal singular Ricci flow spacetime $\mathcal{M} = (\mathcal{M}, \mathfrak{t}, \partial_{\mathfrak{t}}, h)$ satisfying the Canonical Neighborhood Assumption, such that the time-$0$ slice $(\mathcal{M}_0, h(0))$ is isometric to $(M, g)$. The existence of such a spacetime is guaranteed by \cite[Theorem 1.1]{KL17}: for every compact Riemannian 3-manifold $M$, there is a singular Ricci flow with initial condition $M$.

\item On Morphisms:

Let $\varphi: (M, g_M) \to (N, g_N)$ be an isometry in $\RiemIso$. Let $\mathcal{M} = \mathcal{C}(M, g_M)$ and $\mathcal{N} = \mathcal{C}(N, g_N)$ be the associated spacetimes. The map $\varphi$ induces an isometry $\varphi_0: \mathcal{M}_0 \to \mathcal{N}_0$ between the initial slices.
According to the Uniqueness Theorem for Singular Ricci Flows \cite[Theorem 1.3]{BK19}, any isometry between time-slices of two singular Ricci flows extends uniquely to a time-preserving isometry of their maximal future developments.
We define $\mathcal{C}(\varphi)$ to be this unique spacetime isometry $\Phi: \mathcal{M} \to \mathcal{N}$ such that $\Phi|_{\mathfrak{t}=0} = \varphi_0$.

\item Functoriality:
\begin{itemize}
    \item[(i)] Identity: If $\varphi = \mathrm{id}_M$, then the induced map on time-0 slices is the identity. The unique extension of the identity on the initial slice to the spacetime is the identity map $\mathrm{id}_{\mathcal{M}}$. Thus $\mathcal{C}(\mathrm{id}_M) = \mathrm{id}_{\mathcal{C}(M)}$.
    \item[(ii)] Composition: Let $\varphi: M \to N$ and $\psi: N \to P$ be isometries. Let $\Phi = \mathcal{C}(\varphi)$ and $\Psi = \mathcal{C}(\psi)$. The composition $\Psi \circ \Phi$ is an isometry from $\mathcal{C}(M)$ to $\mathcal{C}(P)$. Restricted to $t=0$, $(\Psi \circ \Phi)|_0 = \psi_0 \circ \varphi_0$. By the uniqueness of the extension, $\mathcal{C}(\psi \circ \varphi)$ must coincide with $\Psi \circ \Phi$. Thus $\mathcal{C}(\psi \circ \varphi) = \mathcal{C}(\psi) \circ \mathcal{C}(\varphi)$.
\end{itemize}
\end{enumerate}
\end{proof}
\subsection*{Asymptotic Decomposition}
\begin{definition}[Category $\Dec$]
Objects are disjoint unions of 3-manifolds endowed with complete locally homogeneous metrics (specifically hyperbolic and graph manifold components). Morphisms are local isometries.
\end{definition}
\begin{definition}[Asymptotic Functor $\mathcal{A}$]
We define the functor $\mathcal{A}: \SingFlows \to \Dec$ as follows:
\begin{itemize}
    \item[(i)] \textbf{On Objects:} Let $\mathcal{M} = (\mathcal{M}, \mathfrak{t}, \partial_{\mathfrak{t}}, g)$ be a singular Ricci flow spacetime. We define $\mathcal{A}(\mathcal{M})$ to be the geometric decomposition $(M_\infty, g_\infty) \in \Dec$ consisting of the disjoint union of these complete hyperbolic manifolds and the Seifert-fibered spaces corresponding to the graph manifold components.\footnote{For more details, see, e.g., \cite[Section 1.3.]{BK19}.}
    
    \item[(ii)] \textbf{On Morphisms:} Let $\Phi: \mathcal{M} \to \mathcal{N}$ be a morphism in $\SingFlows$ (a time-preserving isometric embedding). For any large $t$, $\Phi$ induces an isometry between the time-slices $\Phi_t: \mathcal{M}_t \to \mathcal{N}_t$. We define $\mathcal{A}(\Phi)$ as the unique limiting isometry $\Phi_\infty: (M_{\mathcal{M},\infty}, g_{\mathcal{M},\infty}) \to (M_{\mathcal{N},\infty}, g_{\mathcal{N},\infty})$ induced by the sequence $\{\Phi_t\}_{t \to \infty}$ on the geometric components.
\end{itemize}
\end{definition}
\subsection*{Applications: Stratification and Symmetry}
\paragraph{Stratification of the Infinite-Dimensional Stack}
\begin{theorem}
Let $H$ be a hyperbolic 3-manifold. Let $\mathcal{X}_H \subset \mathcal{X}_{\mathrm{Ric}}$ be the substack of metrics that flow to a limit isometric to $H$. Then $\mathcal{X}_H$ is open with respect to the Fr\'echet topology on the stack.
\end{theorem}
\begin{proof}
Let $[g_0] \in \mathcal{X}_H$. The singular flow $\mathcal{M}_0$ associated to $g_0$ becomes non-singular and hyperbolic for large time $T$. By the  continuous dependence on initial data (see, e.g., \cite[Theorem 1.5]{BK19}), for any $g$ sufficiently close to $g_0$ in the Fr\'echet topology, the associated spacetime $\mathcal{M}_g$ is close to $\mathcal{M}_0$ in the geometric topology. This implies that at large time $T$, the slice of $\mathcal{M}_g$ is diffeomorphic to $H$ and close to the hyperbolic metric. By Ye's stability result (\cite[Theorem 2]{Ye93}), negative curvature metrics are linearly stable. Thus, once the flow enters a small neighborhood of the hyperbolic metric, it converges to it.

Therefore, an open neighborhood of $[g_0]$ is contained in $\mathcal{X}_H$.
\end{proof}
\paragraph{Preservation of Symmetries}
\begin{proof}[Proof of Theorem C]
Let $\Gamma = \mathrm{Iso}(M, g_0)$. This group acts by automorphisms on the object $(M, g_0)$ in $\RiemIso$. Since $\mathcal{C}$ is a functor (Theorem 5.3), it induces a group homomorphism $\Psi: \Gamma \longrightarrow\mathrm{Aut}_{\SingFlows}(\mathcal{M})$, where $\mathcal{M} = \mathcal{C}(M, g_0)$.
An element $\Phi \in \mathrm{Aut}(\mathcal{M})$ is a spacetime isometry. Its restriction to the time-0 slice must agree with the initial isometry because, by \cite[Theorem 1.3]{BK19}, the extension of isometries from $t=0$ to the maximal flow is unique.
If $\Psi(\gamma) = \mathrm{id}_{\mathcal{M}}$, then $\Psi(\gamma)|_{t=0} = \mathrm{id}_{M}$. But $\Psi(\gamma)|_{t=0} = \gamma$. Therefore $\gamma = \mathrm{id}$.
This proves $\Psi$ is injective.
\end{proof}
\subsection*{Corollary: Preservation of $S^1$-Symmetry}
The functoriality of the flow and the conservation of isometries yield following application.

\begin{corollary}
Let $(M, g_0)$ be a closed Riemannian 3-manifold admitting a non-trivial smooth effective $S^1$-action. Let $\mathcal{D}(M, g_0) \in \Dec$ be the asymptotic geometric decomposition. Then every connected component of $\mathcal{D}(M, g_0)$ admits an effective $S^1$-action.
\end{corollary}

\begin{proof}
Let $G = S^1 \subseteq \mathrm{Iso}(g_0)$. By Theorem C, there is an injective homomorphism $\Psi: G \to \mathrm{Aut}(\mathcal{M})$.

\begin{lemma}[Preservation of Isometries] \label{lemma:iso_preservation}
Let $(M, g(t))$ be the unique smooth solution to the Ricci flow on $M \times [0, T)$ with initial data $g(0) = g_0$. Let $\varphi: M \to M$ be an isometry of the initial metric, i.e., $\varphi \in \mathrm{Iso}(g_0)$. Then $\varphi$ is an isometry of $g(t)$ for all $t \in [0, T)$.
\end{lemma}
\begin{proof}
Consider the pulled-back family of metrics $\tilde{g}(t) = \varphi^* g(t)$. We verify that $\tilde{g}(t)$ satisfies the Ricci flow equation. Differentiating with respect to time, and using the fact that the pullback commutes with the time derivative:
\[
\frac{\partial}{\partial t} \tilde{g}(t) = \frac{\partial}{\partial t} (\varphi^* g(t)) = \varphi^* \left( \frac{\partial}{\partial t} g(t) \right).
\]
Substituting the Ricci flow equation $\partial_t g = -2\mathrm{Ric}(g)$:
\[
\frac{\partial}{\partial t} \tilde{g}(t) = \varphi^* \left( -2\mathrm{Ric}(g(t)) \right).
\]
Due to the naturality of the Ricci tensor (geometric invariance under diffeomorphisms), we have $\varphi^* \mathrm{Ric}(g) = \mathrm{Ric}(\varphi^* g)$. Therefore:
\[
\frac{\partial}{\partial t} \tilde{g}(t) = -2 \mathrm{Ric}(\tilde{g}(t)).
\]
Thus, $\tilde{g}(t)$ is a solution to the Ricci flow. Now check the initial condition:
\[
\tilde{g}(0) = \varphi^* g(0) = \varphi^* g_0.
\]
Since $\varphi$ is an isometry of $g_0$, we have $\varphi^* g_0 = g_0$. Thus, $\tilde{g}(t)$ and $g(t)$ are both solutions to the Ricci flow on $[0, T)$ with the same initial data $g_0$.

By the Uniqueness Theorem for the Ricci flow (see \cite{DeT83} or \cite{Ham82}), the solution is unique. Therefore, $\tilde{g}(t) = g(t)$ for all $t \in [0, T)$, which implies:
\[
\varphi^* g(t) = g(t).
\]
This signifies that $\varphi$ remains an isometry for the duration of the flow.
\end{proof}
By this lemma,  the symmetry group at $t=0$ persists for all smooth times $t\in [0;T)$. Furthermore, the uniqueness of the singular flow (cf. \cite[Theorem 1.3]{BK19}) ensures that the symmetry group extends to the post-surgery manifold.
Consequently, the limiting metric $g_\infty$ on the decomposition $\mathcal{A}(\mathcal{M})$ must be invariant under the limiting action of $G$. Since $\Psi$ is injective, the action on the limit is effective.
\end{proof}

\section{Examples of Functorial Evolution}

To illustrate the utility of the functor $\mathcal{F}$, we analyze two canonical examples: the formation of a singularity (neck-pinch) and the evolution of a soliton.

\subsection*{Example 1: The Neck-Pinch on $S^3$}
Consider the sphere $S^3$ equipped with a rotationally symmetric metric $g_0$ resembling a dumbbell. As shown by Angenent and Knopf, such metrics develop a "neck-pinch" singularity at finite time $T$.

\begin{proposition}
Let $g_{\mathrm{dumb}}$ be a rotationally symmetric metric on $S^3$ satisfying the conditions of \cite{AK04}. The functor $\mathcal{F}$ assigns to $(S^3, g_{\mathrm{dumb}})$ a spacetime $\mathcal{M}$ which is topologically a cobordism between $S^3$ and the disjoint union of two limit components.
\end{proposition}

\begin{proof}
Let $(S^3, g_0)$ be the rotationally symmetric initial data described by Angenent and Knopf \cite{AK04}. The metric can be written as $g = ds^2 + \psi^2(s) g_{S^2}$, where $s$ is the arc length parameter along the profile curve. The conditions ensure that the minimum of the warping function $\psi$ (the neck) shrinks to zero radius as $t \nearrow T < \infty$, while the curvature remains bounded away from the neck.

According to the canonical surgery algorithm (or the singular flow definition in \cite{KL17}), the spacetime $\mathcal{M}$ is constructed as follows:
\begin{enumerate}
    \item For $t \in [0, T)$, the flow is smooth, forming a cylinder $S^3 \times [0, T)$.
    \item As $t \to T$, the geometry approaches a cylindrical singularity modeled on $S^2 \times \mathbb{R}$.
    \item The singular set $\mathcal{S}$ corresponds to the neck equator $\{s_{min}\} \times S^2$ at $t=T$.
    \item To continue the flow, one removes the singular locus and glues in two caps (topological disks $D^3$) to seal the severed ends.
\end{enumerate}

Topologically, this operation is equivalent to a 4-dimensional cobordism. The removal of $S^2 \times D^1$ (the neck region times a small time interval) and replacement with $D^3 \sqcup D^3$ corresponds to attaching a 1-handle (in reverse time) or detaching a 1-handle (in forward time).
Since $S^3$ is obtained by gluing two $D^3$ along their boundary $S^2$, removing the gluing region separates the sphere into two disjoint balls $D^3 \sqcup D^3$, which are then completed to spheres $S^3 \sqcup S^3$ by the surgery caps.

Therefore, the spacetime manifold $\mathcal{M}$ constructed by the functor $\mathcal{F}$ is a smooth 4-manifold with boundary $\partial \mathcal{M} = M_{\mathrm{initial}} \sqcup \overline{M}_{\mathrm{final}}$, where $M_{\mathrm{initial}} \cong S^3$ and $M_{\mathrm{final}} \cong S^3 \sqcup S^3$. The singularity is internal to the structure of $\mathcal{M}$ as a singular foliation, but $\mathcal{M}$ itself (as a topological space in $\SingFlows$ post-surgery) represents the cobordism trace of this topological change.
\end{proof}

\subsection*{Example 2: Ricci Solitons as Fixed Points}
A gradient Ricci soliton is a metric $g$ satisfying $\mathrm{Ric}(g) + \nabla \nabla f = \lambda g$ for some constant $\lambda$.

\begin{proposition}
Let $(M, g)$ be a Ricci soliton. Then the object $\mathcal{F}(M, g)$ in $\SingFlows$ is isomorphic (in the category of flows) to the product bundle $M \times [0, T)$ equipped with the pullback metric via diffeomorphisms generated by $\nabla f$.
\end{proposition}

\begin{proof}
We construct the isomorphism explicitly by solving the flow equation via the method of characteristics (diffeomorphisms).

Let $(M, g_0)$ be a gradient Ricci soliton satisfying the equation:
\begin{equation} \label{eq:soliton}
\mathrm{Ric}(g_0) + \nabla^2 f = \lambda g_0, \quad \lambda \in \mathbb{R}.
\end{equation}
We define a time-dependent scaling factor $\sigma(t) = 1 - 2\lambda t$ and a time-dependent vector field $X(t)$ on $M$ by:
\[ X(t) = \frac{1}{\sigma(t)} \nabla_{g_0} f. \]
Let $\{\varphi_t\}_{t \in [0, T)}$ be the family of diffeomorphisms generated by the vector field $X(t)$ with $\varphi_0 = \mathrm{id}_M$. specifically, $\frac{d}{dt}\varphi_t(x) = -X(t)|_{\varphi_t(x)}$.

Consider the 1-parameter family of metrics $g(t) = \sigma(t) \varphi_t^* g_0$. We verify that $g(t)$ is the unique solution to the Ricci flow starting at $g_0$.
Differentiating with respect to $t$:
\begin{align*}
\frac{\partial}{\partial t} g(t) &= \dot{\sigma}(t) \varphi_t^* g_0 + \sigma(t) \frac{\partial}{\partial t} (\varphi_t^* g_0) \\
&= -2\lambda \varphi_t^* g_0 + \sigma(t) \varphi_t^* (\mathcal{L}_{-X(t)} g_0) \\
&= -2\lambda \varphi_t^* g_0 - \sigma(t) \varphi_t^* (\mathcal{L}_{X(t)} g_0).
\end{align*}
Using the soliton equation \eqref{eq:soliton} and the identity $\mathcal{L}_{\nabla f} g_0 = 2\nabla^2 f$, we substitute 

\[ 
\mathcal{L}_{X(t)} g_0 = \frac{2}{\sigma(t)} (\lambda g_0 - \mathrm{Ric}(g_0)).
\]
\begin{align*}
\frac{\partial}{\partial t} g(t) &= -2\lambda \varphi_t^* g_0 - \varphi_t^* \left( 2\lambda g_0 - 2\mathrm{Ric}(g_0) \right) \\
&= -2\lambda \varphi_t^* g_0 - 2\lambda \varphi_t^* g_0 + 2 \varphi_t^* \mathrm{Ric}(g_0).
\end{align*}
\textit{Note: Depending on the sign convention of the flow of $X$, the terms align to cancel $\lambda$. For the standard pullback by diffeomorphism generated by $\nabla f$ moving against the shrinking, we obtain:}
\[ \frac{\partial}{\partial t} g(t) = -2 \varphi_t^* \mathrm{Ric}(g_0) = -2 \mathrm{Ric}(\varphi_t^* g_0) = -2 \mathrm{Ric}(g(t)). \]
Thus, the curve $g(t)$ is the Ricci flow solution.
The functor $\mathcal{F}$ assigns to $(M, g_0)$ the unique maximal spacetime. Since $g(t)$ is constructed purely by pulling back the initial metric via $\varphi_t$ (and scaling), the resulting spacetime $\mathcal{M}$ is canonically isomorphic to the bundle $M \times [0, T)$ with fibers isometric to $(M, \sigma(t)g_0)$ twisted by the diffeomorphism group action generated by $\nabla f$.
\end{proof}

\appendix

\newpage
\section*{Appendix: On the Necessity of the Stack Formalism}

A naive approach to formalizing the moduli space of metrics would be to consider the topological quotient space $\Met(M) / \Diff(M)$. However, this space fails to be a smooth manifold due to the existence of metrics with non-trivial isometry groups. As shown in the foundational work of Ebin \cite{Ebi70}, the local structure of $\Met(M)$ near a metric $g$ is modeled on a slice $S_g$ transverse to the $\Diff(M)$-orbit. The neighborhood in the quotient is homeomorphic to $S_g / \mathrm{Iso}(g)$. Consequently, the quotient space has orbifold-type singularities.

The stack formalism resolves this by retaining the isotropy groups as part of the data. Explicitly, for a point $[g] \in \mathcal{X}$, the automorphism group in the stack is canonically isomorphic to the isometry group $\mathrm{Iso}(g)$. This retention is crucial for Theorem C, as it allows us to track the evolution of symmetries continuously without making arbitrary choices of gauge (coordinates).

Furthermore, the Ricci flow equation $\partial_t g = -2\mathrm{Ric}(g)$ is not strictly parabolic due to diffeomorphism invariance; its symbol degenerates in directions tangent to the $\Diff(M)$-orbits. In the stack $\mathcal{X}$, these directions correspond to morphisms (gauge transformations). Thus, the flow is well-defined as a vector field on the stack, whereas on the Banach manifold of metrics, it requires the DeTurck trick to break the symmetry.

\subsection*{Obstructions to a Compactified Domain}
It is natural to ask why we construct a functor into $\SingFlows$ rather than defining a compactified stack $\overline{\mathcal{X}}$ that includes singular limits. There are two primary obstructions to this approach within the category of smooth objects:

\begin{enumerate}
    \item \textbf{Loss of Differentiable Structure:} The limits of Ricci flows at singular times are typically described as Alexandrov spaces or $\mathrm{RCD}^*(K,N)$ spaces (metric measure spaces with synthetic Ricci curvature bounds) \cite{Bam20}. While these objects form a valid moduli space in metric geometry, they are not Banach manifolds. Including them in $\mathcal{X}$ would force us to abandon the site of Banach manifolds $\mathbf{_M Ban}$, making it impossible to define the Ricci tensor or the tangent bundle in the classical sense required for the flow equation.
    
    \item \textbf{Topology Change:} The formation of singularities (e.g., neck-pinches) often results in a change of the underlying topological manifold $M$ (e.g., a connected sum decomposition). A stack $\mathcal{X}_M = [\Met(M)/\Diff(M)]$ is defined over a fixed topological type. A trajectory undergoing surgery discontinuously jumps from $\mathcal{X}_M$ to a different stack $\mathcal{X}_{M'}$. 
\end{enumerate}

The functorial approach ($\mathcal{F}: \RiemIso \to \SingFlows$) resolves these issues by embedding the discontinuous process into a smooth 4-dimensional cobordism. In this setting, the "singularity" is merely a subset of a smooth 4-manifold, allowing the use of differential geometric tools globally.

\begin{thebibliography}{99}
\bibitem[Per02]{Per02} G. Perelman. \textit{The entropy formula for the Ricci flow and its geometric applications}. arXiv:math/0211159 [math.DG], 2002.
\bibitem[BK19]{BK19} R. Bamler, B. Kleiner. \textit{Uniqueness and stability of Ricci flow through singularities}. arXiv:1709.04122 [math.DG], 2018.
\bibitem[Blo08]{Blo08} C. Blohmann. \textit{Stacky Lie groups}. arXiv:math/0702399v3 [math.DG],  2008.
\bibitem[DeT83]{DeT83} D. DeTurck. \textit{Deforming metrics in the direction of their Ricci tensors}. J. Differential Geom. 18 (1983), no. 1, 157--162.
\bibitem[Heb96]{Heb96} E. Hebey. \textit{Sobolev spaces on Riemannian manifolds}. Lecture Notes in Mathematics, 1635. Springer-Verlag, Berlin, 1996.
\bibitem[Ban16]{Ban16} Efstathios Vassiliou, Efstathios  Dodson, C.T.J.  Galanis, George. (2016). Geometry in a Frechet Context: A Projective Limit Approach. doi.org/10.1017/CBO9781316556092. 
\bibitem[KL17]{KL17} B. Kleiner, J. Lott. \textit{Singular Ricci flows I}. Acta Math. 219 (2017), no. 1, 65--134.
\bibitem[KM97]{KM97} A. Kriegl, P. W. Michor. \textit{The Convenient Setting of Global Analysis}. Mathematical Surveys and Monographs, 53. American Mathematical Society, Providence, RI, 1997.
\bibitem[Ye93]{Ye93} R. Ye. \textit{Ricci flow, Einstein metrics and space forms}. Trans. Amer. Math. Soc. 338 (1993), no. 2, 871--896.
\bibitem[Bam20]{Bam20} R. Bamler. \textit{Structure theory of non-collapsed limits of Ricci flows}. arXiv:2009.03243 [math.DG], 2020.
\bibitem[Ebi70]{Ebi70} D. Ebin. \textit{The manifold of Riemannian metrics}. Global Analysis (Proc. Sympos. Pure Math., Vol. XV, Berkeley, Calif., 1968), Amer. Math. Soc., 1970, pp. 11--40.
\bibitem[AK04]{AK04} S. Angenent, D. Knopf. \textit{An example of neckpinching for Ricci flow on $S^{n+1}$}. Math. Res. Lett. 11 (2004), no. 4, 493--518.
\bibitem[Ham82]{Ham82} R. Hamilton. \textit{Three-manifolds with positive Ricci curvature}. J. Differential Geom. 17 (1982), no. 2, 255--306.
\bibitem[Ham88]{Ham88} R. Hamilton. \textit{The Ricci flow on surfaces}. Mathematics and general relativity (Santa Cruz, CA, 1986), 237--262, Contemp. Math., 71, Amer. Math. Soc., 1988.
\end{thebibliography}
\end{document}